\documentclass[12pt]{article}
\usepackage{latexsym}
\usepackage{amssymb}
\usepackage{epsfig}
\usepackage{amsmath}
\usepackage{amsthm}
\usepackage[mathscr]{eucal}
\usepackage{graphicx}
\usepackage{hyperref}
\usepackage{caption}
\usepackage{subcaption}
\usepackage{enumitem}

\newtheorem{Lemma}{Lemma}

\newtheorem{Theorem}[Lemma]{Theorem}

\newtheorem{Definition}{Definition}

\renewcommand{\qed}{\hfill{\ \ \rule{2mm}{2mm}} \vspace{0.2in}}

\begin{document}

\title{Linear recurrences over a finite field with exactly two periods}
\author{ \textbf{Ghurumuruhan Ganesan}
\thanks{E-Mail: \texttt{gganesan82@gmail.com} } \\
\ \\
Institute of Mathematical Sciences, HBNI, Chennai}
\date{}
\maketitle

\begin{abstract}
In this paper, we study the periodicity structure of finite field linear
recurring sequences whose period is not necessarily maximal
and determine necessary and sufficient conditions for
the characteristic polynomial~\(f\) to have exactly two periods in the sense that the period of any sequence generated
by~\(f\) is either one or a unique integer greater than one.

\vspace{0.1in} \noindent \textbf{Key words:} Linear recurring sequences, characteristic polynomial, exactly two periods.

\vspace{0.1in} \noindent \textbf{AMS 2010 Subject Classification:} 11T71, 12Y05, 68R01;
\end{abstract}

\bigskip

\renewcommand{\theequation}{\thesection.\arabic{equation}}
\setcounter{equation}{0}
\section{Introduction} \label{intro}
Sequences generated from linear recursions  are extensively used in practice due to their periodic nature. The characteristic polynomial which results in finite field sequences with maximal periods have been studied before (Golomb (1967), Lidl and Niederreiter (1986)) in various contexts (see for example  Tsaban and Vishne (2002), Goltvanitsa et al.\ (2012)). Recently Quijada (2015) and Bush and Quijada (2018) have investigated recurrent sequences over rings and finite fields and studied decomposition of the set of sequences into periodic and null sequences and determine the set of all possible least periods associated with polynomials of a given degree and a finite field.

In this paper, we study biperiodic characteristic polynomials that generate sequences whose period is either one or a unique integer larger than one,
irrespective of the initial state. We begin with a discussion of linear recurrent
sequences.
\subsection{Model Description}
Let~\(GF(q)\) denote the finite field containing~\(q\) elements and consider a sequence~\(\{a_n\}_{n \geq 0}\) in~\(GF(q)\) obtained recursively as
\begin{equation}\label{an_rec}
a_n = \sum_{i=1}^{r} c_i a_{n-i},
\end{equation}
where~\(\{c_i\}_{1 \leq i \leq r}, c_r \neq 0\) are elements of~\(GF(q)\) and~\(\mathbf{s}_0 := (a_{-r},a_{-r+1},\ldots,a_{-1})\) is an~\(r\)-tuple forming the \emph{initial state}. The summation in~(\ref{an_rec}) is the addition corresponding to the elements of~\(GF(q)\) and throughout a state always refers to an~\(r\)-tuple with entries in~\(GF(q).\) The set~\({\cal S} := (GF(q))^r\) of all~\(r\)-tuples with entries in~\(GF(q)\) is the \emph{state space}.

In what follows, we derive the cycle structure of the state space~\({\cal S}\) using intuitive arguments and refer to Lemma~\ref{lem_cyc} of Section~\ref{pf_thm1} for formal proofs. Letting~\(\mathbf{s}_1 = (a_{-r+1},\ldots,a_0),\) we say that the state~\(\mathbf{s}_0\) \emph{leads to} the state~\(\mathbf{s}_1\) and denote it as~\(\mathbf{s}_0 \rightarrow \mathbf{s}_1.\) Because the recurrence in~(\ref{an_rec}) is linear, each state~\(\mathbf{s}\) of~\({\cal S}\) leads to a unique state~\(\mathbf{t}\) and is in turn obtained from another unique state~\(\mathbf{u};\) i.e.\ for each~\(\mathbf{s} \in {\cal S}\) there are unique states~\(\mathbf{u}, \mathbf{t}\) such that~\(\mathbf{u} \rightarrow \mathbf{s} \rightarrow \mathbf{t}\)  (see Lemma~\ref{lem_cyc}~\((i)\) of Section~\ref{pf_thm1}).

Suppose~\(\mathbf{s}_0 \rightarrow \mathbf{s}_1 \rightarrow \mathbf{s}_2 \rightarrow \cdots\) is the sequence of states obtained while generating the sequence~\(\{a_n\}_{n \geq 0}\) with initial state~\(\mathbf{s}_0.\) Since the number of states in~\({\cal S}\) is finite, the sequence~\(\{\mathbf{s}_i\}_{i \geq 1}\) repeats itself and so there are integers~\(i,k \geq 1\) such that~\(\mathbf{s}_i = \mathbf{s}_{i+k}\) and all the states~\(\mathbf{s}_j, i \leq j \leq i+k-1\) are distinct. We then define
\begin{equation}\label{cyc_def}
\mathbf{s}_i \rightarrow \mathbf{s}_{i+1} \rightarrow \dots \rightarrow \mathbf{s}_{i+k-1} \rightarrow \mathbf{s}_{i+k} = \mathbf{s}_i
\end{equation}
to be a \emph{cycle} and denote it by~\({\cal C}.\)

Because each state leads to a unique state and is obtained from another unique state, we deduce from the above paragraph that each state of~\({\cal S}\) belongs to a unique cycle and there are disjoint cycles~\({\cal C}_i,1 \leq i \leq L\) such that every state~\(\mathbf{s} \in {\cal S}\) belongs to exactly one of the cycles in~\(\{{\cal C}_i\}_{1 \leq i \leq L}\) (see also Lemma~\ref{lem_cyc}~\((ii)\) of Section~\ref{pf_thm1}). Therefore if~\(\mathbf{s}  = \mathbf{s}_0\) belongs to the cycle~\({\cal C}_i\) containing~\(k\) states, then the resulting sequence of states~\(\mathbf{s}_0 \rightarrow \mathbf{s}_1 \rightarrow \mathbf{s}_2 \rightarrow \cdots\) obtained by applying the recursion~(\ref{an_rec}) is periodic and~\(\mathbf{s}_i = \mathbf{s}_{i+k}\) for all~\(i \geq 0\) (Lemma~\ref{lem_cyc}~\((ii),\) Section~\ref{pf_thm1}).

Due to the periodicity of the state sequence~\(\{\mathbf{s}_i\}_{i \geq 0},\) repetitions occur in the symbol sequence~\(A(\mathbf{s}_0) := \{a_n\}_{n \geq 0}\) as well and so we define the period~\(P(\mathbf{s}_0)\) of the sequence~\(A(\mathbf{s}_0)\) to be the smallest~\(p\) such that there exists an integer~\(l = l(p)\) satisfying~\(a_n = a_{n+p}\) for all~\(n \geq l\) (Lidl and Niederreiter (1986)). Recalling that~\(\mathbf{s}_0\) belongs to the cycle~\({\cal C}_i\) containing~\(k\) states, the period of the symbol sequence~\(A(\mathbf{s}_0)\) must therefore also be~\(k\) (see Lemma~\ref{lem_cyc}~\((iii),\) Section~\ref{pf_thm1}).

Let~\({\cal P}\) be the set of all possible periods of sequences with initial states in~\({\cal S}\) and satisfying recursion~(\ref{an_rec}). The set~\({\cal P}\) is defined to be the \emph{period set} and based on the discussion in the previous paragraphs, we see that
\begin{equation}\label{p_def}
{\cal P} := \{\#{\cal C}_i: 1 \leq i \leq L\}
\end{equation}
where~\(\#{\cal C}_i\) denotes the number of states in the cycle~\({\cal C}_i.\) The zero state~\(\mathbf{0}\) always belongs to the cycle~\(\mathbf{0} \rightarrow \mathbf{0}\) and so~\(1 \in {\cal P}.\) Clearly, the structure of the period set~\({\cal P}\) depends on the coefficients~\(\{c_i\}_{1 \leq i \leq r}\) in the recursion~(\ref{an_rec}). The main question we are interested in is this: What restrictions must be placed on~\(\{c_i\}_{1 \leq i \leq r}\) so that~\({\cal P}\) contains exactly two elements ?

For convenience, we state our results in terms of the \emph{characteristic polynomial} (Golomb (1967)) 
\begin{equation}\label{char_f}
f(x) := 1-\sum_{i=1}^{r} c_i x^{i}.
\end{equation}
Throughout all polynomials have coefficients in~\(GF(q).\)

We have the following definition.
\begin{Definition}
For integer~\(T \geq 2\) we say that the characteristic polynomial~\(f\) is \emph{biperiodic} with \emph{biperiod}~\(T\) if the period set~\({\cal P} = \{1,T\}.\)
\end{Definition}
From~(\ref{p_def}) we see that~\(f\) is biperiodic with biperiod~\(T\) if and only if for each~\(1 \leq i \leq L,\) the cycle~\({\cal C}_i\) has either~\(T\) states or a single state.

To state our main result describing characteristic polynomials with a given biperiod~\(T,\) we begin with a couple of definitions. For integers~\(a,b \geq 1\) we use the notation~\(a \mid b\) to denote that~\(a\) divides~\(b\) and 
let~\(\gcd(a,b)\) denote the greatest common divisor of~\(a\) and~\(b.\) We say that~\(a\) and~\(b\) are relatively prime if~\(\gcd(a,b) = 1\) and let~\(\varphi(a)\) be the Euler function counting the number of positive integers less or equal to~\(a\) that are relatively prime to~\(a.\)

We say that a polynomial~\(g(x)\) is \emph{irreducible} if~\(g\) cannot be written as a product of two polynomials, each with degree strictly larger than one. The \emph{order}~\(\theta(g)\) of~\(g(x)\) is the smallest integer~\(j \geq 1\) such that~\(g(x)\) divides~\(1-x^{j}\) but does not divide~\(1-x^{i}\) for any integer~\(1 \leq i < j\) (Lidl and Niederreiter (1986)). Letting~\(q\) be a power of a prime~\(p\) we
consider two cases below depending on whether the desired biperiod~\(T\) is relatively prime to~\(q\) or not.
\begin{Theorem}\label{uni_per_thm} Let~\(q\) be a power of a prime number and let~\(T \geq 1\) be an integer.
\begin{enumerate}[label=(\roman*)]
\item{Suppose~\(\gcd(T,q) = 1.\) There exists a polynomial~\(f\) with degree~\(r\) and biperiod~\(T\) if and only if
\begin{equation}\label{uni_cond}
T \mid (q^{r-d_0} -1) \text{ and } \varphi(T) \geq r-d_0 \text{ for some } d_0 \in \{0,1\}.
\end{equation}

The condition~(\ref{uni_cond}) holds either for~\(d_0=0\) or~\(d_0 = 1\) but not both and if~(\ref{uni_cond}) holds, then the polynomial~\(f\) is of the form~\[f(x) = (1-x)^{d_0} \cdot g_1(x) \cdot g_2(x) \cdots g_l(x),\]
where~\(l\) is the largest divisor~\(j\) of~\(r-d_0\) such that~\(T \mid (q^{\frac{r-d_0}{j}}-1)\) and the~\(\{g_i\}_{1 \leq i \leq l}\) are distinct irreducible polynomials, each with degree~\(\frac{r-d_0}{l}\) and order~\(T\) and satisfying~\(\prod_{j=1}^{l} g_j(0) = 1.\) Moreover letting~\(N := \frac{q^{r}-q^{d_0}}{T},\) there are~\(L := N+q^{d_0}\) cycles~\(\{{\cal C}_i\}_{1 \leq i \leq L}\) with~\(N\) of the cycles having~\(T\) states each and the remaining~\(q^{d_0}\) cycles having one state each.}

\item{Suppose~\(T = p^{u} \cdot M\) with~\(\gcd(M,q) = 1\) and~\(u \geq 1.\) There exists a polynomial~\(f\) with biperiod~\(T\)
and degree~\(r\) if and only if
\begin{equation}\label{uni_cond2}
M = u = 1 \text{ and } r \leq p.
\end{equation}

If~(\ref{uni_cond2}) holds, then~\(f(x) = (1-x)^{r}\) and letting~\(N := \frac{q^{r}-q}{p},\) there are~\(L := N+q\) cycles~\(\{{\cal C}_i\}_{1 \leq i \leq L}\) with~\(N\) of the cycles having~\(p\) states each and the remaining~\(q\) cycles having one state each.}
\end{enumerate}
\end{Theorem}
As a consequence of Theorem~\ref{uni_per_thm}, a polynomial~\(f\) with degree~\(r \geq 2\) has biperiod~\(q^{r-1}-1\) if and only if~\(f(x)=(1-x)\cdot g(x),\) where~\(g(x)\) is a primitive polynomial of degree~\(r-1\) (i.e.\ the order of~\(g(x)\) is~\(q^{r-1}-1\)) and satisfies~\(g(0)=1.\) In this case, there are exactly~\(2q\) cycles:~\(q\) cycles  each containing~\(q^{r-1}~-~1\) states each and~\(q\) cycles containing one state each. In particular, for~\(q=2\) this characterises all polynomials with the maximum possible biperiod \emph{strictly less} than~\(2^{r}-1.\) Indeed if there existed a biperiod~\(T\) with value between~\(2^{r-1}\) and~\(2^{r}-2,\) then from~(\ref{uni_cond}) we have that~\(T\) must be a factor of~\(2^{r}-1\) and so we would have~\(\gamma \cdot T = 2^{r}-1\) for some integer~\( \gamma \geq 2.\) But this would mean that~\(\gamma\cdot T \geq 2\cdot T \geq 2\cdot 2^{r-1} = 2^{r},\) a contradiction.

We illustrate Theorem~\ref{uni_per_thm} with an example.
\subsection*{Example}
Let~\(q=2\) and~\(r = 4\) so that~\(2^{r}-1 = 15\) and~\(2^{r-1}-1 = 7.\) From condition~(\ref{uni_cond}) we therefore see that the desired biperiod~\(T\) must either be a factor of~\(15\) and satisfy~\(\varphi(T) \geq 4\) or must satisfy~\(T = 7\) and~\(\varphi(T) \geq 3.\) We first consider the case that~\(T\) is a factor of~\(15.\) The non-trivial factors of~\(15\) are~\(3,5\) and~\(15\) and we have~\(\varphi(3) = 2 < r, \varphi(5)  =4 = r\) and~\(\varphi(15) = 2\cdot 4 > r.\) Thus there exists degree~\(4\) polynomials with biperiod~\(5\) and~\(15\) and there does not exist a degree~\(4\) polynomial with biperiod~\(3.\)  In what follows, we give one example of a polynomial in each of the above mentioned valid cases and also describe the corresponding cycle structure of the states.

For~\(T = 5,\) the largest divisor~\(j\) of~\(4\) such that~\(T  \mid (2^{\frac{4}{j}}-1)\) is~\(j=1\) and so any irreducible polynomial of degree~\(4\) and order~\(5\) also has biperiod~\(5.\)  As an example, the polynomial~\(f_1(x) = x^{4} + x^{3}  + x^2 + x + 1\) is irreducible with degree~\(4\) and order~\(5\) and therefore has biperiod~\(5.\) Correspondingly, there are~\(3\) cycles containing~\(5\) states each and one cycle containing one state as enumerated below (states are of the form~\((a_{j-r},\ldots,a_{j-1})\)):\\
\({\cal C}_1 : 0001 \rightarrow 0011 \rightarrow 0110 \rightarrow 1100 \rightarrow 1000 \rightarrow 0001.\)\\
\({\cal C}_2 : 0111 \rightarrow 1111 \rightarrow 1110 \rightarrow 1101 \rightarrow 1011 \rightarrow 0111.\)\\
\({\cal C}_3 : 1001 \rightarrow 0010 \rightarrow 0101 \rightarrow 1010 \rightarrow 0100 \rightarrow 1001.\)\\
\({\cal C}_4 : 0000 \rightarrow 0000.\)

For~\(T = 15,\) any polynomial with degree~\(4\) and order~\(15\) also has biperiod~\(15.\) For illustration, the polynomial~\(f_2(x) = x^{4} + x^{3} + 1\) is irreducible and has order~\(15\) and so correspondingly, there is one cycle containing~\(15\) states and one cycle containing one state as enumerated below (states are of the form~\((a_{j-r},\ldots,a_{j-1})\)):\\
\({\cal C}_1 : 0001 \rightarrow 0010 \rightarrow 0100 \rightarrow 1001 \rightarrow 0011 \rightarrow 0110 \rightarrow 1101 \rightarrow 1010 \rightarrow 0101 \rightarrow 1011 \rightarrow 0111 \rightarrow 1111 \rightarrow 1110 \rightarrow 1100 \rightarrow 1000 \rightarrow 0001.\)\\
\({\cal C}_2 : 0000 \rightarrow 0000.\)

We now study the case~\(T = 7.\) Since~\(\varphi(7) = 6 \geq 3 = r-1,\) there exists a polynomial~\(f\) with biperiod~\(7\) and degree~\(4\) and since~\(r-1 = 3\) is prime, the largest divisor~\(j\) of~\(r-1\) such that~\(7 \mid (2^{\frac{r-1}{j}}-1)\) is one. From Theorem~\ref{uni_per_thm}, we therefore get that~\(f\) is of  the form~\((1+x)\cdot g(x),\) where~\(g(x)\) is irreducible with degree~\(3\) and order~\(7.\) For example, the polynomial~\[f_3(x) = x^{4} + x^{3}  +x^2 + 1 = (x^3+x+1) (x+1)\] has an irreducible factor~\(x^{3}+x+1\) of order~\(7.\) Correspondingly, there are two cycles containing~\(7\) states each and two cycles containing one state each as enumerated below (states are of the form~\((a_{j-r},\ldots,a_{j-1})\)):\\
\({\cal C}_1 : 0001 \rightarrow 0010 \rightarrow 0101 \rightarrow 1011 \rightarrow 0110 \rightarrow 1100 \rightarrow 1000 \rightarrow 0001.\)\\
\({\cal C}_2 : 0111 \rightarrow 1110 \rightarrow 1101 \rightarrow 1010 \rightarrow 0100 \rightarrow 1001 \rightarrow 0011 \rightarrow 0111.\)\\
\({\cal C}_3 : 1111 \rightarrow 1111.\)\\
\({\cal C}_4 : 0000 \rightarrow 0000.\)


The paper is organized as follows. In Section~\ref{pf_thm1}, we collect preliminary results used in the proof of Theorem~\ref{uni_per_thm} and in Section~\ref{pf_uni}, we prove Theorem~\ref{uni_per_thm}.

\renewcommand{\theequation}{\thesection.\arabic{equation}}
\setcounter{equation}{0}
\section{Preliminary results} \label{pf_thm1}
To prove Theorem~\ref{thm1} we collect a few preliminary results in the form of three Lemmas. For completeness we provide small proofs.
We recall that~\({\cal S} = (GF(q))^r\) is the state space consisting of 
all~\(r\)-tuples with entries in~\(GF(q).\)
\begin{Lemma}\label{lem_cyc} Let~\(\{c_i\}_{1 \leq  i \leq r}\) be as in the recursion~(\ref{an_rec}). The following properties hold:
\begin{enumerate}[label=(\roman*)]
\item{For every~\(\mathbf{y} \in {\cal S},\) there exists unique~\(\mathbf{z}_1,\mathbf{z}_2 \in {\cal S}\) such that~\(\mathbf{z}_1 \rightarrow \mathbf{y} \rightarrow \mathbf{z}_2.\)}
\item{There are cycles~\({\cal C}_i,1 \leq i \leq L\) such that every state~\(\mathbf{s} \in {\cal S}\) belongs to exactly one of the cycles in~\(\{{\cal C}_i\}_{1 \leq i \leq L}.\)}
\item{Pick a cycle~\({\cal C}_i, 1 \leq i \leq L\) with~\(p_i\) states and let~\(\{a_n\}_{n \geq 0}\) be the sequence generated with initial state~\(\mathbf{s} \in {\cal C}_i.\) The period~\(P(\mathbf{s})\) of the sequence~\(\{a_n\}_{n \geq 0}\) equals~\(p_i,\) the terms~\(a_n = a_{n+p_i}\) for all~\(n \geq 0\) and the generating function
\begin{equation}\label{gx_def}
G(x) = G(x;\mathbf{s}) := \sum_{n \geq 0} a_n x^{n} = \frac{D(x;\mathbf{s})}{f(x)} = \frac{a(x)}{1-x^{p_i}}
\end{equation}
where~\(f(x) = 1-\sum_{j=1}^{r} c_jx^{j}\) is the characteristic polynomial as defined in~(\ref{char_f}),
\begin{equation}\label{dx_def}
D(x;\mathbf{s}) = \sum_{j=0}^{r-1}d_jx^{j} \text{ with } d_j = \sum_{i=1}^{r-j} a_{-i}c_i, 0 \leq j \leq r-1.
\end{equation}
and~\(a(x) := \sum_{n=0}^{p_i-1}a_n x^{n}.\)}
\end{enumerate}
\end{Lemma}

\begin{proof}[Proof of Lemma~\ref{lem_cyc}]
Property~\((i)\) is true since if~\(\mathbf{y} =(y_1,\ldots,y_r) \in {\cal S}\)  then~\(\mathbf{z}_1 \rightarrow \mathbf{y} \rightarrow \mathbf{z}_2\) where~\(\mathbf{z}_1 = (b,y_1,\ldots,y_{r-1})\) with~\[b := c_r^{-1} \left(y_r - \sum_{i=2}^{r} c_{r-i} y_{i-1}\right)\] and~\(\mathbf{z}_2 = (y_2,\ldots,y_{r},d)\) with~\(d := \sum_{i=1}^{r} c_i y_{r-i}.\)

To prove~\((ii),\) let~\(\mathbf{s}_0 := \mathbf{s}\) be any state. From property~\((i)\) there exists a unique state~\(\mathbf{s}_1\) such that~\(\mathbf{s}_0 \rightarrow \mathbf{s}_1.\) Similarly, there exists a unique state~\(\mathbf{s}_2\) such that~\(\mathbf{s}_1 \rightarrow \mathbf{s}_2.\) Because the number of states in~\({\cal S}\) is finite, we eventually get that~\(\mathbf{s}_j \rightarrow \mathbf{s}_i\) for some integers~\(j \geq 1\) and~\(0 \leq i \leq j-1.\) The integer~\(i =0\) necessarily else we would have~\(\mathbf{s}_{i-1} \rightarrow \mathbf{s}_i\) in addition to~\(\mathbf{s}_j \rightarrow \mathbf{s}_i,\) a contradiction to~\((i).\) By the same argument we get that if~\({\cal C}(\mathbf{s})\) is the cycle containing the state~\(\mathbf{s},\) then for any two distinct states~\(\mathbf{s}_1\) and~\(\mathbf{s}_2\) the corresponding cycles~\({\cal C}(\mathbf{s}_1)\) and~\({\cal C}(\mathbf{s}_2)\) satisfy the following property: Either~\({\cal C}(\mathbf{s}_1) = {\cal C}(\mathbf{s}_2)\) or~\({\cal C}(\mathbf{s}_1) \cap {\cal C}(\mathbf{s}_2) = \emptyset.\)  The distinct cycles in~\(\{{\cal C}(\mathbf{s})\}_{\mathbf{s} \in {\cal S}}\) form the desired set~\({\cal C}_1,\ldots, {\cal C}_L,\) proving~\((ii).\)

To prove~\((iii),\) let~\(\mathbf{s}_0 \rightarrow \mathbf{s}_1 \rightarrow \ldots\)  be the sequence of states generated starting with~\(\mathbf{s}_0 \in {\cal C}_i.\) Recalling that~\(\mathbf{s}_j := (a_{j-r},\ldots,a_{j-1},a_j) = \mathbf{s}_{j+p_i}\) for all~\(j \geq 0,\) we get that~\(a_n = a_{n+p_i}\) for all~\(n \geq 0.\) Thus~\(P(\mathbf{s}) \leq p_i.\) If it so happens that~\(P =  P(\mathbf{s}) \leq p_i-1\) so that~\(a_j = a_{j+P}\) for all~\(j\) sufficiently large, then there exists~\(J\) sufficiently large so that~\(\mathbf{s}_j = \mathbf{s}_{j+P}\) for all~\(j \geq J.\) This implies that~\(\mathbf{s}_j \in {\cal C}_i\) belongs to a cycle of length at most~\(P \leq p_i-1,\) a contradiction. Finally, we substitute the recursion~(\ref{an_rec}) into the definition of~\(G(x)\) in~(\ref{gx_def}) to get the second relation in~(\ref{gx_def}) (see Golomb (1967)). Writing~\[(1-x^{p_i})\cdot G(x) = a(x) + \sum_{n=0}^{\infty} (a_{n+p_i} -a_n) x^{n+p_i}\] and using the fact that the sequence~\(\{a_n\}_{n \geq 0}\) also has period~\(p_i,\) we get the final relation in~(\ref{gx_def}).~\(\qed\)
\end{proof}

The next Lemma relates the period set~\({\cal P}\) to the orders of factors of the characteristic polynomial~\(f(x).\)
\begin{Lemma}\label{thm1}  Let~\(\{c_i\}_{1 \leq  i \leq r}\) be as in the recursion~(\ref{an_rec}) and let~\({\cal C}_i,1 \leq i \leq L\) be the cycle decomposition of the state space~\({\cal S}\) as in Lemma~\ref{lem_cyc}. If~\(p_i\) denotes the number of states in the cycle~\({\cal C}_i,\) then the period set
\begin{equation}\label{per_set_eq}
{\cal P} = {\cal P}(f) := \{p_1,\ldots, p_L\} =  \{\theta(g): g \mid f\}
\end{equation}
is precisely the set of orders of all possible factors of~\(f.\)
\end{Lemma}
We recall that the state~\(\mathbf{0}\) always belongs to the cycle~\(\mathbf{0} \rightarrow  \mathbf{0}\) and so~\(1 \in {\cal P}.\) The rest of the period set is determined from the non-trivial factors of the characteristic polynomial~\(f.\)

As an example, we let~\(q=2\) and consider the characteristic polynomial~\(f_0(x) = 1+x^2 + x^4 = (1+x+x^2)^2.\) The order of the factor~\(1+x+x^2\) of~\(f_0\) is~\(3\) and the order of~\(f_0\) is~\(6\) and so the period set for the sequences generated by~\(f_0\) is~\({\cal P} = \{1,3,6\}.\) Indeed, the degree of~\(f_0\) is~\(4\) and the sixteen states in~\(\{0,1\}^{4}\) are present in the following four cycles where each state~\(\mathbf{s}\) is of the form~\((a_{j-r},a_{j-r+1},\ldots,a_{j-1}):\)\\
\({\cal C}_1 : 0001 \rightarrow 0010 \rightarrow 0101 \rightarrow 1010 \rightarrow 0100 \rightarrow 1000 \rightarrow 0001.\)\\
\({\cal C}_2 : 0011 \rightarrow 0111 \rightarrow 1111 \rightarrow 1110 \rightarrow 1100 \rightarrow 1001 \rightarrow 0011.\)\\
\({\cal C}_3 : 0110 \rightarrow 1101 \rightarrow 1011 \rightarrow 0110.\)\\
\({\cal C}_4 : 0000 \rightarrow 0000.\)

The period set in Lemma~\ref{thm1} could also be studied using minimal polynomials (see for example, Chapter~\(6\) of Lidl and Niedereitter (1986)). For completeness, we give a small proof.

\begin{proof}[Proof of Lemma~\ref{thm1}]: Let~\({\cal P} = \{p_1,\ldots,p_L\}\) be the period set as defined in~(\ref{per_set_eq}) where~\(p_i\) denotes the number of states in the cycle~\({\cal C}_i\) and let~\({\cal Q} := \{\theta(g) : g \mid f\}\) be the set of orders of all possible factors of the characteristic polynomial~\(f.\) We first see that~\({\cal P} \subseteq {\cal Q}.\)  Let~\(\mathbf{s} = (a_{-r},\ldots,a_{-1})\) be any state in~\({\cal S}\) and let~\(\{a_n\}_{n \geq 0}\) be the corresponding sequence generated using~(\ref{an_rec}). If~\(G(x)\) is the generating function of~\(\{a_n\}_{n\geq 0}\) then cancelling the common factors in the second relation of~(\ref{gx_def}) and using the final relation of~(\ref{gx_def}), we get that~\(G(x) = \frac{u(x)}{b(x)} = \frac{a(x)}{1-x^{W}}\) where~\(b(x)\) is a factor of~\(f(x)\) and~\(W\) is the period of~\(\{a_n\}_{n \geq 0}.\)
Thus~\((1-x^{W})\cdot u(x) = b(x) \cdot a(x)\) and since~\(u(x)\) and~\(b(x)\) have no factors in common,~\(b(x)\) divides~\(1-x^{W}.\) Consequently, the order~\(\theta\) of~\(b(x)\) is at most~\(W.\) Next using the fact that~\[(1-x^{\theta})\cdot G(x) = \sum_{n=0}^{\theta-1}a_nx^{n} +\sum_{n=\theta}^{\infty}(a_{n}-a_{n-\theta}) x^{n}\] has finite degree, we get that~\(a_n = a_{n-\theta}\) for all~\(n\) sufficiently large. Thus the period~\(W\) of~\(\{a_n\}_{n \geq 0}\) is at most~\(\theta.\) Thus~\(W = \theta\) and since~\(W\) is also the number of states in the cycle~\({\cal C}(\mathbf{s})\) (see Lemma~\ref{lem_cyc}), this completes the proof of~\({\cal P} \subseteq {\cal Q}.\)

To see that~\({\cal Q} \subseteq {\cal P},\) let~\(g\) be a factor of~\(f\) so that~\(g(x)\cdot h(x) = f(x)\) for some polynomial~\(h(x) = \sum_{j=0}^{r-1}h_jx^{j}\) of degree at most~\(r-1.\) We first see that choosing~\(\mathbf{s} = (a_{-r},\ldots,a_{-1})\) appropriately, the polynomial~\(D(x;\mathbf{s})\) defined in~(\ref{gx_def}) equals~\(h(x).\)
Indeed we solve~(\ref{dx_def}) backwards; first we choose~\(a_{-1}\) such that~\(a_{-1} = c_r^{-1} \cdot h_{r-1}\) and having determined~\(a_{-1},\ldots,a_{-i+1},\) we set~\[a_{-i} := c_r^{-1}\left(h_{r-i}-\sum_{j=1}^{i-1}a_{-j}c_{r-j}\right).\] The generating function of the resulting sequence~\(\{a_n\}_{n \geq 0}\) then equals~\(\frac{1}{g(x)}.\)  If~\(\omega\) is the number of states in the corresponding cycle~\({\cal C}(\mathbf{s}),\) then from Lemma~\ref{lem_cyc} we get that the sequence~\(\{a_n\}_{n \geq 0}\) also has period~\(\omega\) and arguing as in the previous paragraph, we get that~\(\omega = \theta(g),\) the order of the polynomial~\(g.\)~\(\qed\)
\end{proof}

The final ingredient used in the proof of Theorem~\ref{uni_per_thm} concerns the order of powers of polynomials and number of polynomials with a given degree. We say that a polynomial~\(g(x) = \sum_{i=0}^{w} c_i x^{i}, c_w \neq 0\) is \emph{monic} if the leading coefficient~\(c_w  =1\) and have the following result.
\begin{Lemma}\label{lidl_lemma}
\begin{enumerate}[label=(\alph*)]
\item{If~\(g\) is an irreducible polynomial of order~\(a\) with~\(g(0) \neq 0,\) then for integer~\(b \geq 1\) the polynomial~\(g^{b}\) has order~\(a \cdot p^{t},\) where~\(t\) is the smallest integer such that~\(p^{t} \geq b.\)}

\item{ Let~\(e \geq 2\) be an integer satisfying~\(\gcd(e,q) = 1.\) For integer~\(m \geq 1,\) let~\(\delta(m,e)\) be the number of monic irreducible polynomials  of degree~\(m\) and order~\(e.\) We have that~\(\delta(m,e) = \frac{\varphi(e)}{m}\) if~\(m\) is the smallest integer~\(j \geq 1\) such that~\(q^{j} \equiv 1 \mod{e}.\) For all other~\(m,\) we have that~\(\delta(m,e) =0.\)}
\end{enumerate}
\end{Lemma}
\begin{proof}[Proof of Lemma~\ref{lidl_lemma}] 
Part~\((a)\) follows from Theorem~\(3.8\) of Lidl and Niederreiter (1986) and part~\((b)\) follows from Theorem~\(3.5\) of Lidl and Niederreiter (1986).~\(\qed\)
\end{proof}

\renewcommand{\theequation}{\thesection.\arabic{equation}}
\setcounter{equation}{0}
\section{Proof of Theorem~\ref{uni_per_thm}} \label{pf_uni}
\begin{proof}[Proof of Theorem~\ref{uni_per_thm}\((i)\)] Let
\begin{equation}\label{f_type2}
f = (1-x)^{d_0} \cdot g_1^{d_1}\cdots  g_t^{d_t}
\end{equation}
be biperiodic of degree~\(r\) and biperiod~\(T\) where~\(\{g_i\}_{1 \leq i \leq t}\) are distinct irreducible polynomials having order~\(\theta_i :=\theta(g_i)  \geq 2.\) Since~\(f(0) = 1\) we must have that~\(\prod_{j=1}^{t}g_j^{d_j}(0) = 1\) and so in particular we get that~\(g_j(0) \neq 0\) for each~\(1 \leq j  \leq t.\)

From Lemma~\ref{thm1}, we know that~\(\{1,\theta_1,\ldots,\theta_t\} \subseteq {\cal P},\) the period set. But since~\(f\) is biperiodic with biperiod~\(T,\) the period set~\({\cal P} = \{1,T\}\) and so we must have that the~\(g_i\)'s all have the same order~\(T.\) Moreover, if~\(d_i \geq 2\) for some~\(i \geq 1,\) then from Lemma~\ref{lidl_lemma}\((a),\) we have that the order of~\(g_i^{d_i}\) is at least~\(T\cdot p\) and so we must have~\(d_i=1\) for all~\(1 \leq i \leq t.\) Similarly, if~\(d_0 \geq 2\) then the order of~\((1-x)^{d_0}\) is of the form~\(p^{u}\) for some integer~\(u \geq 1.\) Since~\(\gcd(T,q) = \gcd(T,p) = 1,\) we have that~\(T \neq p^{u}\) and so we must have that~\(d_0 \in \{0,1\}.\)

From the above paragraph we get that~\(f = (1-x)^{d_0} \cdot g_1\cdots  g_t\) where\\\(d_0 \in \{0,1\}\) and all the~\(g_i\) are irreducible and have the same order~\(T.\) Also since~\(f(0) =1\) we must have that~\(\prod_{j=1}^{t}g_j(0) = 1.\) For convenience we now write
\begin{equation}\label{f_type}
f = (1-x)^{d_0} \cdot \alpha \cdot h_1 \cdots h_t
\end{equation}
where~\(\alpha \in GF(q)\) and each~\(h_i\) is monic. From Lemma~\ref{lidl_lemma}\((b)\) we know that there exists a monic irreducible polynomial of order~\(T \geq 2\) and degree~\(m\) if and only if~\(m\) is the smallest integer~\(j\) such that~\(T \mid (q^{j}-1).\) For such an~\(m,\) all the polynomials~\(\{h_i\}_{1 \leq i \leq t}\) must have the same degree~\(m\) which in turn must equal~\(\frac{r-d_0}{t}.\) Thus in effect, the integer~\(m\) must satisfy two conditions:\\
\((1)\)~\(m\) is the smallest integer~\(j\) such that~\(T \mid (q^{j}-1).\)\\
\((2)\)~\(m\) is a divisor of~\(r-d_0.\)

We show that an integer~\(m\) satisfying conditions~\((1)-(2)\) above exists if and only if~\(T \mid (q^{r-d_0}-1).\) Indeed if~\(m\) satisfies~\((1)-(2),\) then~\(m\) divides~\(r-d_0\) and so~\(r-d_0 = a\cdot m\) for some integer~\(a.\) We then have that~\[q^{r-d_0} = q^{am} = (q^{m})^{a} \equiv 1 \mod {T}.\] Conversely, suppose that~\(T \mid (q^{r-d_0}-1)\) and let~\(m\) be the smallest integer~\(j\) such that~\(T \mid (q^{j}-1).\) Thus~\(m\) satisfies condition~\((1)\) above and we would like to see that~\(m\) also satisfies~\((2).\) Indeed if~\(m\) does not divide~\(r-d_0,\) we write~\(r-d_0 = mw + s\) for~\(1 \leq s \leq m-1\) and get that~\[q^{r-d_0} = \left(q^{m}\right)^{w}\cdot q^{s}  \equiv q^{s} \mod {T}.\] Since~\(q^{r-d_0} \equiv 1 \mod{T},\) we then get that~\(q^{s} \equiv 1 \mod{T},\) a contradiction to the definition of~\(m.\)

So far we have shown that if~\(f\) is biperiodic and of the form~(\ref{f_type}), then all the polynomials~\(h_i\) must have the same degree~\(m = \frac{r-d_0}{t}.\) Such an integer~\(m\) exists if and only if~\(T \mid (q^{r-d_0}-1)\) and moreover since~\(r\) and~\(r-1\) have no common factors, there cannot exist an integer~\(m\) that is a divisor of both~\(r-1\) and~\(r.\) Consequently, the above condition \(T \mid (q^{r-d_0}-1)\) holds either for~\(d_0=0\) or~\(d_0=1\) but not both.

Setting~\(\frac{r-d_0}{t} = m,\)  we again use Lemma~\ref{lidl_lemma}\((b)\) to get that the number of monic irreducible polynomials of order~\(T\) and degree~\(\frac{r-d_0}{t}\) equals~\(\varphi(T)\cdot \frac{t}{r-d_0}.\) Since the number of polynomials in~(\ref{f_type}) equals~\(t,\) we must have~\(\varphi(T) \cdot \frac{t}{r-d_0} \geq t\) which is possible if and only if~\(\varphi(T) \geq r-d_0.\) This proves~(\ref{uni_cond}).

Assuming~(\ref{uni_cond}) holds, let~\(\{{\cal C}_i\}_{1 \leq i \leq L}\) be the set of all resulting cycles as determined in Lemma~\ref{thm1}, each having either~\(T\) states or a single state. A cycle~\({\cal C}_i\) is of the  form~\(\mathbf{s} \rightarrow \mathbf{s} = (a_{-r},a_{-r+1},\ldots,a_{-1})\) if and only if~\[a_{-r} = a_{-r+1} = \ldots = a_{-1} =  \beta\] for some~\(\beta \in GF(q).\)  Therefore if~\(d_0 = 0\)  then~\(f(1) = 1-\sum_{i=1}^{r}c_i \neq 0\) and so the only cycle with containing a single state is the cycle~\(\mathbf{0} \rightarrow \mathbf{0}.\) If~\(d_0 = 1,\) then~\(\sum_{i=1}^{r}c_i = 1\) and there are~\(q\) cycles of the form~\(\beta \mathbf{1} \rightarrow \beta \mathbf{1}, \beta \in GF(q)\) each containing one state, where~\(\mathbf{1}\) denotes the all ones state~\((1,1,\ldots,1).\)~\(\qed\)

\end{proof}

\begin{proof}[Proof of Theorem~\ref{uni_per_thm}\((ii)\)] We first eliminate the case~\(M \geq 2\) by contradiction. Assume there exists a biperiodic polynomial~\(f\) with degree~\(r\) and biperiod~\(T\) of the form~(\ref{f_type2}) where~\(\{g_i\}_{1 \leq i \leq t}\) are distinct irreducible polynomials having order~\(\theta(g_i) \geq 2.\) If~\(d_0 \geq 2,\) then from Lemma~\ref{lidl_lemma}\((a),\) we get that the order of~\((1-x)^{d_0}\) is of the form~\(p^{b}\) for some~\(b \geq 1.\) From Lemma~\ref{thm1} we know that~\(p^{b} \in {\cal P} = \{1,T\}\) and this is a contradiction since~\(T = p^{u} \cdot M\) where~\(M \geq 2\) is not a multiple of~\(p.\) Thus~\(d_0 \leq 1\) and arguing as in the proof of Theorem~\ref{uni_per_thm}\((i)\) above, we have that~\(d_i = 1\) for all~\(1 \leq i \leq t\) and that each~\(g_i\) has order~\(\theta(g_i) = T.\)

Next if~\(g_i\) has degree~\(m_i,\) then its order~\(T \mid (q^{m_i}-1)\) and so~\[q^{m_i}-1 = p^{n_i}-1 \equiv 0 \mod T\]
for some integer~\(n_i \geq 1.\) In particular this means that~\(T\) cannot contain~\(p\) as a divisor since~\(\gcd(p,p^{n_i}-1) = 1\) and this is a contradiction. This completes the case~\(M \geq 2.\)

Now suppose~\(T = p^{u}\) and let~\(f\) be any biperiodic polynomial with degree~\(r\) and biperiod~\(T\) of the form~(\ref{f_type2}) where~\(\{g_i\}_{1 \leq i \leq t}\) are distinct irreducible polynomials having order~\(\theta(g_i) \geq 2.\) If~\(d_i \geq 1\) for some~\(i,\) then from Lemma~\ref{thm1} we get that~\(\theta(g_i) \in {\cal P}\) and arguing as in the previous paragraph we get that~\(\gcd(\theta(g_i),p) = 1.\) This is a contradiction since~\({\cal P} = \{1,p^{u}\}.\) Thus~\(d_i  = 0\) for all~\(1 \leq i \leq t\) and so~\(f = (1-x)^{d_0}.\) Moreover, since~\(f\) is of degree~\(r,\) we must have that~\(d_0 =r.\)

From Lemma~\ref{lidl_lemma}\((a)\) we have that the order of~\((1-x)^{r}\) is~\(p^{b}\) where~\(b\) satisfies~\(p^{b-1} < r \leq p^{b}.\)
If~\(b \geq 2,\) then there exists an integer~\(e_0 \leq r-1\) such that~\(p^{b-2} < e_0 \leq p^{b-1}\) and so~\((1-x)^{e_0}\) has order~\(p^{b-1}.\)
By Lemma~\ref{thm1}, we then have that both~\(p^{b}\) and~\(p^{b-1}\) belong to~\({\cal P},\) a contradiction. Thus~\(b=1\)
and we must have that~\(1 < r \leq p.\) In this case the characteristic polynomial is~\(f(x) = (1-x)^{r}\) with period set~\(\{1,p\}.\) Arguing as in the last paragraph of proof of Theorem~\ref{uni_per_thm}~\((i),\) we get that there are~\(N\) cycles containing~\(p\) states each and~\(q\) cycles containing one state each.~\(\qed\)

\end{proof}

\subsubsection*{Acknowledgement}
I thank Professors V. Arvind, C. R. Subramanian and the referees for crucial comments that led to an improvement of the paper. I also thank IMSc for my fellowships.

\bibliographystyle{plain}

\end{document}